\documentclass[11pt,english]{article}
\usepackage[margin= 2 cm,bottom=23mm, top= 20mm]{geometry}

\usepackage{amsthm}
\usepackage{amsmath}
\usepackage{amssymb}
\usepackage{setspace}
\usepackage{mathtools}
\usepackage{verbatim}
\usepackage{csquotes}
\usepackage[hidelinks]{hyperref}
\usepackage{thm-restate}
\usepackage{cleveref}
\usepackage{enumitem}
\setlist[itemize]{leftmargin=*} 
\setlist[enumerate]{leftmargin=*}
\usepackage{framed}
\usepackage{floatrow}
\usepackage[T1]{fontenc}
\usepackage{bbm}
\usepackage[color=Orange!50!white, textwidth=22mm]{todonotes}

\usepackage{mathdots}
\usepackage{xcolor}
\usepackage{diagbox}
\usepackage{colortbl}

\usepackage{graphicx}
\usepackage{subcaption}

\theoremstyle{plain}

\newtheorem{theorem}{Theorem}[section]
\newtheorem{claim}[theorem]{Claim}

\newtheorem{lemma}[theorem]{Lemma}

\newtheorem{problem}[theorem]{Problem}
\theoremstyle{definition}

\newtheorem{defn}[theorem]{Definition}
\newtheorem*{defn*}{Definition}

\expandafter\def\expandafter\normalsize\expandafter{%
    \normalsize
    \setlength\abovedisplayskip{4pt}
    \setlength\belowdisplayskip{4pt}
    \setlength\abovedisplayshortskip{4pt}
    \setlength\belowdisplayshortskip{4pt}
}

\usepackage[square,sort,comma,numbers]{natbib}
\newcommand{\cupdot}{\mathbin{\mathaccent\cdot\cup}}

\newcommand{\Bin}{\mathrm{Bin}}

\newcommand{\cA}{\mathcal{A}}

\newcommand{\bit}{\mathrm{bit}}

\newcommand{\tw}{\mathrm{tw}}
\newcommand{\ex}{\mathrm{ex}}

\renewcommand{\Pr}{\mathbb{P}}

\newcommand{\floor}[1]{
    \left \lfloor #1 \right \rfloor
}

\DeclareMathOperator*{\argmax}{arg\,max}

\title{Ramsey numbers of hypergraphs with a given size}
\author{Domagoj Brada\v{c}\thanks{Department of Mathematics, ETH, Z\"urich, Switzerland. Research supported in part by SNSF grant 200021\_196965. Email: \textbf{\{domagoj.bradac, benjamin.sudakov\}@math.ethz.ch}.}
\and Jacob Fox\thanks{Department of Mathematics, Stanford University, Stanford, CA. Email: \textbf{jacobfox@stanford.edu}. Research
supported by NSF Awards DMS-1953990 and DMS-2154129.}
\and Benny Sudakov\footnotemark[1]}
\date{}

\begin{document}

\maketitle

\begin{abstract}
The $q$-color Ramsey number of a $k$-uniform hypergraph $H$ is the minimum integer $N$ such that any $q$-coloring of the complete $k$-uniform hypergraph on $N$ vertices contains a monochromatic copy of $H$. 
The study of these numbers is one of the central topics in Combinatorics. In 1973, Erd\H{o}s and Graham asked to maximize the Ramsey number of a graph as a function of the number of its edges.
Motivated by this problem, we study the analogous question for hypergaphs. For fixed $k \ge 3$ and $q \ge 2$ we prove that the largest possible $q$-color Ramsey number of a $k$-uniform hypergraph with $m$ edges is at most $\tw_k(O(\sqrt{m})),$ where $\tw$ denotes the tower function. We also present a construction showing that this bound is tight for $q \ge 4$. This resolves a problem by Conlon, Fox and Sudakov. They previously proved the upper bound for $k \geq 4$ and the lower bound for $k=3$.  Although in the graph case the tightness follows simply by considering a clique of appropriate size, for higher uniformities the construction is rather involved and is obtained by using paths in expander graphs. 
\end{abstract}

\section{Introduction}
    For a $k$-uniform hypergraph $H$ and a positive integer $q,$ we denote by $r_k(H; q)$ the $q$-color Ramsey number of $H$ defined as the minimum integer $N$ such that any $q$-coloring of the complete $k$-uniform hypergraph on $N$ vertices, denoted by $K_N^{(k)}$, contains a monochromatic copy of $H$. When $H = K^{(k)}_n,$ we simply write $r_k(n; q)$. The existence of these numbers was famously shown by Ramsey \cite{ramsey} in 1930. Since then, finding good bounds on $r_k(H;q)$ for various (hyper)graphs $H$ has been one of the most major areas of study in Discrete mathematics. The first important results in this direction were exponential bounds on the so-called diagonal graph Ramsey number, namely that $\sqrt{2}^n < r_2(n; 2) < 4^n,$ where the upper bound was proven by Erd\H{o}s and Szekeres \cite{erdos-szekeres} and the lower bound by Erd\H{o}s \cite{erdos} as one of the first applications of the probabilistic method. Both of these arguments easily extend to give similar bounds for any fixed number of colors $q$. 
    Despite a great amount of interest and the fact that these bounds are at least 70 years old, until very recently they have been only improved by lower order terms. In March 2023, a major breakthrough was obtained by
    Campos, Griffiths, Morris and Sahasrabudhe \cite{CGMS}, who improved the upper bound to $(4-\epsilon)^n$.

    In the case of hypergraphs, Erd\H{o}s and Rado \cite{erdos-rado} showed that for some constant $c = c(k, q),$  the Ramsey numbers satisfy $r_k(n; q) \le \tw_k(cn),$ where $\tw_k(x)$ denotes the tower function defined as $\tw_1(x) = x$ and $\tw_k(x) = 2^{\tw_{k-1}(x)}$ for $k\ge 2.$ On the other hand, an ingenious construction of Erd\H{o}s and Hajnal (see e.g. \cite{graham1991ramsey}), known as the stepping-up lemma, allows one to obtain a lower bound for hypergraphs of uniformity $k+1$ from lower bounds for uniformity $k,$ essentially gaining an extra exponential. However, this construction only works if the number of colors, $q,$ is at least $4$ or the uniformity, $k$, is at least $3.$ In particular, we have the bounds $r_k(n; 2) \ge \tw_{k-1}(c n^2)$ and $r_k(n; 4) \ge \tw_k(cn).$ The first bound comes from applying a random construction for uniformity $3$ and then applying the stepping-up lemma. Erd\H{o}s, Hajnal and Rado \cite{erdos-hajnal-rado} conjectured that $r_3(n;2) > 2^{2^{cn}}$, which would, by the stepping-up lemma, imply $r_k(n;2) \ge \tw_k(c_kn),$ thus determining the correct tower height of these numbers. However, this remains a major open problem.

    Given the difficulty of finding good bounds for complete graphs and hypergraphs, Burr and Erd\H{o}s \cite{burr-erdos} initiated the study of Ramsey numbers of sparse graphs and, in particular, conjectured that for any integer $\Delta,$ there is $c(\Delta)$ such that $r_2(G; 2) \le c(\Delta) n$ for any $n$-vertex graph with maximum degree at most $\Delta.$ This conjecture was proven by Chv\'{a}tal, R\"{o}dl, Szemer\'edi and Trotter \cite{chvatal-rodl-sz-tr} using Szemer\'edi's celebrated regularity lemma. The development of the hypergraph regularity lemma lead to the generalization of this result to bounded degree hypergraphs proven by Cooley, Fountoulakis, K\"{u}hn and Osthus \cite{cooley2009embeddings}. Burr and Erd\H{o}s also made the stronger conjecture that $r_2(G; 2) < c(d) n$ should also hold for all $d$-degenerate graphs on $n$ vertices and it was proved by Lee \cite{lee2017ramsey}. 

    In 1973, Erd\H{o}s and Graham \cite{erdos-graham} posed a natural question of maximizing the Ramsey number of a graph as a function of the number of its edges.
    Since Ramsey numbers of sparse graphs grow slowly, it is natural to guess that in order to maximize the Ramsey number of a graph with $m$ edges, one should make it as dense as possible. This has motivated Erd\H{o}s and Graham to conjecture that among all graphs with $m = \binom{n}{2}$ edges, the complete graph $K_n$ has the largest Ramsey number. This conjecture appears extremely difficult and there has been no real progress on it. Therefore Erd\H{o}s \cite{erdos-edges-conj} made a weaker conjecture that there is a constant $c$ such that $r_2(G; 2) \le 2^{c \sqrt{m}}$ for any graph $G$ with $m$ edges and no isolated vertices, which would be sharp by the above-mentioned lower bound for the Ramsey number of the complete graph. This conjecture has been resolved by Sudakov \cite{sudakov2011conjecture}. In contrast to many results mentioned above, the argument in \cite{sudakov2011conjecture} only works for two colors and it would be interesting to extend this result to more colors.

In this paper we consider the Erd\H{o}s and Graham question for hypergaphs. Naively one might expect that for fixed $k$ and $q,$ there exists $c = c(k, q)$ such that any $k$-uniform hypergraph $H$ with $m$ edges satisfies $r_k(H; q) \le \tw_k(c m^{1/k})$, i.e., the complete hypergraph is a maximizer. This however, was shown to be false by Conlon, Fox and Sudakov \cite{conlon2009ramsey} who constructed a $3$-uniform hypergraph with $m$ edges whose $4$-color Ramsey number at least $2^{2^{c \sqrt{m}}}$ for some positive absolute constant $c$. On the other hand, they showed that any $k$-uniform hypergraph $H$ with $m$ edges satisfies $r_k(H; q) \le \tw_k(c \sqrt{m}),$ for $k \ge 4,$ while $r_k(H; q) \le \tw_k(c \sqrt{m} \log m)$ for $k=3,$ where the constant $c$ depends only on $k$ and $q$. In a survey on graph Ramsey theory \cite{conlon2015recent}, they further asked whether it is possible to remove the logarithmic factor for the $3$-uniform case.

    \begin{problem} \label{prob:upper-bound-3}
        Show that for any $q \ge 2,$ there exists $c_q$ such that $r_3(H; q) \le 2^{2^{c_q \sqrt{m}}}$ for any $3$-uniform hypergraph $H$ with $m$ edges and no isolated vertices.
    \end{problem}

    In the present paper, we resolve Problem~\ref{prob:upper-bound-3}. Moreover, our proof extends to larger uniformities as well so we present a unified proof for all $k \ge 3$.

    \begin{theorem} \label{thm:upper-bound}
        For any $k \ge 3,$ and any fixed number of colors $q,$ there is a constant $C_{k, q}$ such that the $q$-color Ramsey number of any $k$-uniform hypergraph $H$ with $m$ edges and no isolated vertices is at most $\tw_k(C_{k,q} \sqrt{m}).$         
    \end{theorem}

We also provide a construction showing that the above bound is tight up to a constant factor in front of $\sqrt m$ . 
Although in the graph case the tightness follows simply by considering a clique of appropriate size, for higher uniformities the construction is rather involved and is obtained by using the paths in expander graphs. Due to our reliance on the stepping-up lemma, the construction requires $4$ colors.

    \begin{theorem} \label{thm:lower-bound}
        For any $k \ge 2,$ there exist a constant $c_k > 0$ such that for any positive integer $m$ there is a $k$-uniform hypergraph with $m$ hyperedges and no isolated vertices whose $4$-color Ramsey number is at least $\tw_k(c_k \sqrt{m}).$
    \end{theorem}

The rest of this short paper is organized as follows. In Section 2 we prove Theorem~\ref{thm:upper-bound} and in Section 3 we prove Theorem~\ref{thm:lower-bound}. We systematically ignore floor and ceiling signs whenever they are not crucial for the argument. In the use of asymptotic notation we sometimes omit the dependence on the uniformity, $k$, and the number of colours, $q$, since we treat them as constants.
    
    \section{Proof of Theorem~\ref{thm:upper-bound}}
    Before presenting the proof, let us give a brief outline. The main new idea is to show that every hypergraph with $m$ edges has a strong coloring (see Definition~\ref{def:strong-coloring}) with $t = O(\sqrt{m})$ colors such that the product of the sizes of the color classes is $2^{O(\sqrt{m})}.$ Given a colored complete $k$-uniform hypergraph $G$ on $\tw_k(C \sqrt{m})$ vertices, we then apply Erd\H{o}s and Rado's upper bound on hypergraph Ramsey numbers mentioned in the introduction along with a simple supersaturation argument to find many monochromatic cliques of size $t$ in $G$. Then, it is enough to find a set of cliques of one colour which form a complete $t$-partite hypergraph with part sizes corresponding to the color classes of the given strong coloring. This will follow from a version of the hypergraph extension of the K\H{o}v\'{a}ri-S\'{o}s-Tur\'{a}n theorem \cite{kovari-sos-turan}. Such an extension was first proved by Erd\H{o}s \cite{erdos-hypergraph-kst}. In our setting, the number of parts and their sizes are allowed to grow with the size of the hypergraph. The aforementioned result of Erd\H{o}s does not provide such a bound, though it can easily be extracted from most of the known proofs.
        
    Let $K_{s_1, \dots, s_t}$ denote the complete $t$-partite $t$-uniform hypergraph with part sizes $s_1, \dots, s_t$ and denote by $\ex(n, K_{s_1, \dots, s_t})$ the maximum number of edges in a $t$-uniform hypergraph not containing $K_{s_1, \dots, s_t}$ as a subgraph.

    We require an upper bound on $\ex(n, K_{s_1, \dots s_t}),$ where the number of parts and their sizes are allowed to grow with $n.$ Such an upper bound is surely widely known, but we have not found a reference which contains the bound we need. Hence, we include the short proof for completeness. Note that the exponent of $n$ in our bound is not best possible, however, it is sufficient for our purposes and allows for a cleaner proof.
\begin{lemma} \label{lem:kovari-sos-turan}
    Let $s_1, \dots, s_t$ be positive integers and denote $P = \prod_{i=1}^t s_i.$ Then, for all $n \ge 1,$
    \[ \ex(n, K_{s_1, \dots, s_t}) < Pn^{t - P^{-1}}. \]
\end{lemma}
\begin{proof}   
    We prove the statement by induction on $t.$ For $t=1,$ the claim is trivial. Assume now $t \ge 2$ and let $H$ be a $t$-uniform hypergraph with $m \ge P n^{t - P^{-1}}$ edges. We need to show that $H$ contains a copy of $K_{s_1, \dots, s_t}.$ For $W \in \binom{V(H)}{s_t},$ let
    \[ N(W) = \left\{ f \in \binom{V(H)}{t-1} \, \vert \, f \cup \{w\} \in E(H), \forall w \in W \right\}. \]
    For $f \in \binom{V(H)}{t-1},$ let $d(f)$ denote the number of edges of $H$ containing the set $f.$ Double counting, we have
    \[ \sum_{W \in \binom{V(H)}{s_t}} |N(W)| = \sum_{f \in \binom{V(H)}{t-1}} \binom{d(f)}{s_t}. \]
    Using that $\binom{x}{s}$ is convex and $\sum_{f \in \binom{V(H)}{t-1}} d(f) = tm,$ we can apply Jensen's inequality to obtain
    
    \[ \sum_{W \in \binom{V(H)}{s_t}} |N(W)| \ge \binom{n}{t-1} \binom{tm / \binom{n}{t-1}}{s_t} \ge 
    \frac{n \cdots (n-t+2)}{(t-1)!} \cdot \left(\frac{t! m}{s_t \cdot n \cdots (n-t+2)}\right)^{s_t} \ge \frac{m^{s_t}}{s_t^{s_t} \cdot n^{(t-1)(s_t-1)}}. \]
    Denoting $P' = \prod_{i=1}^{t-1} s_i = P / s_t,$ by the pigeonhole principle, there is a set $W$ with
    \[ |N(W)| \ge \binom{n}{s_t}^{-1} \cdot \frac{m^{s_t}}{s_t^{s_t} \cdot n^{(t-1)(s_t-1)}} \ge \frac{(P n^{t-P^{-1}})^{s_t}}{s_t^{s_t} \cdot n^{(t-1)(s_t-1) + s_t}} \ge \frac{s_t^{s_t} P' n^{ts_t - (P')^{-1}}}{s_t^{s_t} \cdot n^{ts_t - t + 1}} = P' n^{t-1-(P')^{-1}}. \] 
    By the induction hypothesis, the $(t-1)$-uniform hypergraph formed by the edge set $N(W)$ contains a copy of $K_{s_1, \dots, s_{t-1}},$ which together with $W$ forms $K_{s_1, \dots, s_t},$ as required.
\end{proof}

\begin{defn} \label{def:strong-coloring}
    A \emph{strong coloring} of a hypergraph $H$ is a partition of $V(H)$ into \emph{color classes} $V_1, \dots, V_t$ such that every edge of $H$ contains at most one vertex from each of the sets $V_1, \dots, V_t.$
\end{defn}

\begin{lemma} \label{lem:coloring}
    Fix $k \ge 2$ and let $H$ be a $k$-uniform hypergraph with $m$ edges and no isolated vertices. Then, there is a strong coloring $V_1 \cupdot \dots \cupdot V_t$ of $H$ such that $t = O(\sqrt{m})$ and moreover, $\prod_{i=1}^t |V_i| \le 2^{O(\sqrt{m})}.$
\end{lemma}
\begin{proof}
    We first partition the vertices of $H$ according to their degree as follows. Set $\Delta \coloneqq \sqrt{m}$ and denote $s = \floor{\log_2 \Delta} + 1.$ Let $U_0 = \{ v \in V(H) \, \vert \,  d(v) > \Delta \}$ and for $1 \le i \le s,$ let $U_i = \{ v \in V(H) \, \vert \, \Delta / 2^i < d(v) \le \Delta / 2^{i-1} \}.$ Since $V(H)$ has no isolated vertices, it is clear that $U_0 \cupdot \dots \cupdot U_s$ is a partition of $V(H).$ We color each of the sets $U_i$ using distinct colors. Each vertex in $U_0$ receives a distinct color. For $i \ge 1,$ the vertices in $U_i$ are greedily colored one by one using at most $t_i \coloneqq k \Delta / 2^{i-1}$ colors. This is possible since having colored some vertices in $U_i,$ the next vertex $v \in U_i$ to be colored shares an edge with at most $(k-1) d(v) \le (k-1) \Delta / 2^{i-1} \le t_i - 1$ previously colored vertices in $U_i.$ 
    
    Let $V_1, \dots, V_t$ denote the color classes produced by the coloring described above. It remains to verify that it satisfies the desired properties. To this end, for $0 \le i \le s,$ we denote $n_i = |U_i|$ and $m_i = \sum_{v \in U_i} d(v).$ Clearly, $\sum_{i=0}^s m_i = \sum_{v \in V(H)} d(v) = km$ and for $0 \le i \le s,$ we have $km \ge m_i \ge n_i \cdot \Delta / 2^i.$ In particular, $n_0 \le km / \Delta = k \sqrt{m}$.
    
    The number of colors used, $t,$ satisfies
    \[ t \le n_0 + \sum_{i=1}^s t_i \le k \sqrt{m} + \sum_{i=1}^s k\sqrt{m} / 2^{i-1} = O(\sqrt{m}). \]
    
    Recall that $t_i = k \Delta / 2^{i-1}$ and $n_i \le 2^i m_i / \Delta \le 2^i km / \Delta.$ By the AM-GM inequality, the product of the sizes of the color classes used to color $U_i, i \ge 1$ is at most 
    \[ \left(\frac{n_i}{t_i}\right)^{t_i} < \left( \frac{4^i m}{\Delta^2} \right)^{k \Delta / 2^{i-1}} = (4^i)^{k\Delta / 2^{i-1}} =2^{k\Delta i/2^i}. \]

    Multiplying the above bound for all $1 \le i \le s$ and using that the series $\sum_{i=1}^\infty i / 2^i$ converges, we obtain
    \[ \prod_{j=1}^t |V_j| \le \prod_{i=1}^s 2^{k\Delta i/2^i} = 2^{O_k(\sqrt{m})}. \]
    
\end{proof}

As mentioned in the outline, we will use the following bound on the Ramsey number of a complete $k$-uniform hypergraph.
\begin{theorem}[Erd\H{o}s, Rado \cite{erdos-rado}] \label{thm:erdos-rado}
    For positive integers $q, k$ there is a constant $C' = C'(q, k)$ such that $r_k(n; q) \le \tw_k(C'n).$
\end{theorem}

\begin{proof}[Proof of Theorem~\ref{thm:upper-bound}]
    Let $m = e(H)$ and let $N = \tw_k(C \sqrt{m})$ where $C = C(k, q)$ is a large constant to be chosen implicitly later. Consider an arbitrary $q$-coloring of the complete $k$-uniform hypergraph on $N$ vertices and call this colored hypergraph $G$. We need to show that $G$ contains a monochromatic copy of $H.$
    
    Let $V_1 \cupdot \dots \cupdot V_t = V(H)$ be a strong coloring of $H$ satisfying $t \le O(\sqrt{m})$ and $P \coloneqq \prod_{i=1}^t |V_i| \le 2^{O(\sqrt{m})}$ given by Lemma~\ref{lem:coloring}. We remark that $P$ will correspond to the value of $P$ in our application of Lemma~\ref{lem:kovari-sos-turan}. We denote $s_i = |V_i|$ for $i \in [t].$ Let $R = r_k(t; q) \le \tw_k(O(\sqrt{m})),$ where the bound holds by Theorem~\ref{thm:erdos-rado}. A standard supersaturation argument allows us to find many monochromatic copies of $K^{(k)}_t$ in $G$ of the same color. Indeed, by definition, every set of $R$ vertices of $G$ contains a monochromatic copy of $K^{(k)}_t$. On the other hand, any copy of $K^{(k)}_t$ is contained in $\binom{N-t}{R-t}$ sets of $R$ vertices of $G$. Putting these two facts together and applying the pigeonhole principle, there is a color, say red, such that the number of red copies of $K^{(k)}_t$ in $G$ is at least
    \begin{equation}
        \binom{N}{R} / \left(q \binom{N - t}{R-t} \right) \ge \left( \frac{N}{R} \right)^t / q. \label{eq:edges-gamma}
    \end{equation}
    We construct an auxiliary $t$-uniform hypergraph $\Gamma$ on the vertex set $V(G)$ where a $t$-set forms an edge if it forms a red $t$-clique in $G.$ Provided that $e(\Gamma) > \ex(N, K^{(t)}_{s_1, \dots, s_t}),$ there must exist a copy of $K^{(t)}_{s_1, \dots, s_t}$ in $\Gamma$ which corresponds to a red complete $t$-partite $k$-uniform hypergraph with part sizes $s_1, \dots, s_t$ in $G$ and by the existence of the strong coloring $V(H) = V_1 \cupdot \dots \cupdot V_t,$ it contains a red subgraph isomorphic to $H.$
    
    It remains to ensure that $e(\Gamma) > \ex(N, K^{(t)}_{s_1, \dots, s_t}).$ Recall that $P = \prod_{i=1}^{t} |V_i| = 2^{O(\sqrt{m})}$. Hence, by \eqref{eq:edges-gamma} and Lemma~\ref{lem:kovari-sos-turan}, it is enough to show that
    \[ \left(\frac{N}{R} \right)^t / q > N^{t - 2^{-O(\sqrt{m})}}, \]
    or equivalently,
    \[ R^t q < N^{2^{-O (\sqrt{m})}}. \]
    It will be convenient to compare the logarithms of the two sides. We remind the reader that $\tw_k(x) = 2^{\tw_{k-1}(x)}$ for $k \ge 2.$ Thus, we have
    \[ \log_2(R^t q) = t \log_2(R) + O(1) = O(\sqrt{m}) \cdot \tw_{k-1}(O(\sqrt{m})) = \tw_{k-1}(O(\sqrt{m})), \]
    where in the last inequality we used that $k \ge 3.$ On the other hand,
    \[ \log_2(N^{2^{-O(\sqrt{m})}}) = \log_2(N) \cdot 2^{-O(\sqrt{m})} = \tw_{k-1}(C \sqrt{m}) \cdot 2^{-O(\sqrt{m})} \ge \tw_{k-1}(C/2 \cdot \sqrt{m}), \]
    where in the last inequality we used that $k\ge 3$ and chose $C$ to be large enough compared to the implicit constant in the $O$ notation. It follows that for large enough $C,$ we have
    $R^t q < N^{2^{-O(\sqrt{m})}},$ as needed.
\end{proof}

\section{Proof of Theorem~\ref{thm:lower-bound}}
In this section, we prove Theorem~\ref{thm:lower-bound}. We shall start with a few definitions which are used in the proof and present a variant of the step-up coloring that we use. After that, we give an informal discussion of the main ideas behind the proof and then we present the proof itself.

\subsection{Setup}
To begin, we recall an important function used in this construction. For a nonnegative integer $x,$ let $x = \sum_{i=0}^{\infty} a_i 2^i$ be its unique binary representation (where $a_i = 0$ for all but finitely many $i$). We denote $\bit(x, i) = a_i.$ Then $\delta(x, y) \coloneqq \max \{ i \in \mathbb{Z}_{\ge 0} \, \vert \, \bit(x, i) \neq \bit(y, i)\}.$ For nonnegative integers $x_1 < x_2 < \dots < x_t,$ we denote $\delta(\{x_1, \dots, x_t\}) = (\delta_1, \dots, \delta_{t-1})$ where for $i \in [t-1],$ $\delta_i = \delta(x_i, x_{i+1}).$ The following properties of this function are well known and easy to verify.

\begin{enumerate}[label=P\arabic*)]
    \item \label{prop:delta-smaller} $x < y \iff \bit(x, \delta(x, y)) < \bit(y, \delta(x, y))$.
    \item \label{prop:not-equal} For any $x < y < z,$ $\delta(x, y) \neq \delta(y, z)$.
    \item \label{prop:maximum} For any $x_1 < x_2 < \dots < x_k,$ $\delta(x_1, x_k) = \max_{1 \le i \le k-1} \delta(x_i, x_{i+1})$. 
\end{enumerate}

Let us now define the coloring which will be used to prove Theorem~\ref{thm:lower-bound}. For a positive integer $n$, we start with a red-blue coloring $\phi_n^{(2)}$ of the complete graph with vertex set $\{0, \dots, N_2-1\}$, where $N_2 = N_2(n) = 2^{n/2}$, containing no monochromatic clique of size $n.$ Such a coloring exists by the well known result of Erd\H{o}s mentioned in the introduction. For $k \ge 3,$ the coloring $\phi_n^{(k)}$ is on the vertex set $\{0, \dots, N_k - 1\},$ where $N_k = N_k(n) = 2^{N_{k-1}(n)} = \tw_k(n)$ and is defined as follows. For a set $\{x_1, \dots, x_k\}$ with $0 \le x_1 < \dots < x_k < N_k,$ we consider the vector $\delta(\{x_1, \dots, x_k\}) = (\delta_1, \dots, \delta_{k-1}).$ 
Note that $0 \le \delta_i < N_{k-1}$ for all $i \in [k-1]$. Hence for distinct $\delta_i$, the set $\{\delta_1, \dots, \delta_{k-1}\}$ forms an edge of the complete $(k-1)$-uniform hypegraph on $[N_{k-1}]$ with color $\phi^{(k-1)}_n(\{\delta_1, \dots, \delta_{k-1}\})$. For $k = 3,$ the $4$-coloring is given as:
\[
    \phi^{(3)}_n(\{x_1, x_2, x_3\}) = 
    \begin{cases}
        C_1, &\text{if } \delta_1 < \delta_2 \text{ and } \phi_n^{(2)}(\{\delta_1, \delta_2\}) \text{ is red;}\\
        C_2, &\text{if } \delta_1 < \delta_2 \text{ and } \phi_n^{(2)}(\{\delta_1, \delta_2\}) \text{ is blue;}\\
        C_3, &\text{if } \delta_1 > \delta_2 \text{ and } \phi_n^{(2)}(\{\delta_1, \delta_2\}) \text{ is red;}\\
        C_4, &\text{if } \delta_1 > \delta_2 \text{ and } \phi_n^{(2)}(\{\delta_1, \delta_2\}) \text{ is blue.}
    \end{cases}
\]

We denote by $\argmax_{i \in [k-1]} \delta_i$ the unique index $j \in [k-1]$ such that $\delta_j = \max_{i \in [k-1]} \delta_i$, where the uniqueness follows from Properties~\ref{prop:not-equal}~and~\ref{prop:maximum}. For $k \ge 4,$ the coloring is given as:

\[ 
    \phi^{(k)}_n(\{x_1, \dots, x_k\}) = 
\begin{cases}
    \phi^{(k-1)}_n(\{\delta_1, \dots, \delta_{k-1}\}), &\text{if } \delta \text{ is a monotone sequence;}\\
    C_1, &\text{if } \delta \text{ is not monotone and } \argmax_{i \in [k-1]} \delta_i \in \{1, k-1\}; \\
    C_2, &\text{if } \argmax_{i \in [k-1]} \delta_i \not\in \{1, k-1\}.
\end{cases}
\]

\subsection{Proof outline}
Let us now discuss the main ideas of our proof. First, we recall Erd\H{o}s and Hajnal's proof of the lower bound $r_k(n; 4) \ge \tw_k(2^{-k} n).$ Their proof uses a slightly different coloring than given above, but the same proof works with our coloring, so we consider it instead. Suppose that $\phi^{(k)}_n$ contains a monochromatic clique of size $n_k = 2^k n$. Denote by $x_1 < x_2 < \dots < x_{n_k}$ the vertices of this clique and let $\delta = (\delta_1, \dots, \delta_{n_k-1}) = \delta(\{x_1, \dots, x_{n_k}\}).$ It is not difficult to show that $\delta$ must contain a monotone contiguous subsequence $\delta' = (\delta_a, \delta_{a+1}, \dots, \delta_b)$ of length at least $n_k/2.$ By Property~\ref{prop:maximum} and the definition of $\phi^{(k)}_n,$ it follows that $\{\delta_a, \delta_{a+1}, \dots \delta_b\}$ forms a monochromatic clique in $\phi^{(k-1)}_n$ of size at least $n_k/2.$ Applying the same argument to this clique in $\phi^{(k-1)}_n$, we find a monochromatic clique of size at least $n_k / 4$ in $\phi^{(k-2)}_n$ and so on. After $k-2$ steps, we thus reach a monochromatic clique of size $4n$ in $\phi^{(2)}_n,$ a contradiction.

We will show that instead of a clique, we can take a much sparser hypergraph $H_k$ on $n_k = \alpha_k n$ vertices with $m = O(n_k^2) = O_k(n^2)$ edges and reach a similar conclusion, i.e. that $\phi^{(k-1)}_n$ contains a monochromatic copy of some $(k-1)$-uniform hypergraph $H_{k-1}$ on $\alpha_{k-1} n$ vertices, where $H_{k-1}$ is ``of the same form'' as $H_k.$ For the argument to work, we need to make sure of a few a things. With $x_1 < x_2 < \dots < x_{n_k}$ and $\delta = (\delta_1, \dots, \delta_{n_{k-1}})$ defined as above, we need that $\delta$ contains a monotone subsequence of length $\Omega(n_k).$ Furthermore, this monotone subsequence should imply the existence of a hypergraph $H_{k-1}$ on $\alpha_{k-1}n$ vertices on which we can apply induction. We remark here that $H_{k-1}$ will not be a fixed hypergraph, but rather some large enough hypergraph of the same form as $H_k$. Finally, after $k-2$ steps, we should reach a graph containing a clique of size $n$ to obtain a contradiction. Given that this argument works for a clique and we want a much sparser hypergraph, it should be no surprise that our construction is based on an expander graph. We next define our construction formally and carry out the outlined proof strategy.

\subsection{Formal proof}
Given a graph $G$ and an integer $k \ge 2,$ we define a $k$-uniform hypergraph $H = H(G, k)$ on the same vertex set where for every path $(v_1, \dots, v_{k-1})$ in $G$ and any vertex $v_k \in V(G)$ not on this path, we put the $k$-edge $\{v_1, \dots, v_k\}$ in $H$. Note that for $k=2,$ $H(G, k)$ is simply the complete graph on the vertex set $V(G).$


Given a $k$-uniform hypergraph $\Gamma$ with a coloring $\phi \colon E(\Gamma) \rightarrow \mathcal{C}$ and a hypergraph $H,$ a set of vertices $X \subseteq V(\Gamma)$ forms a monochromatic copy of $H$ if there exists a bijection $\Psi \colon V(H) \rightarrow X$ such that $\phi(\Psi(e)) = \phi(\Psi(e'))$ for all $e, e' \in E(H).$



We shall need the following simple lemma about sparse random graphs.

\begin{lemma} \label{lem:edge-distribution}
    For any $d \ge 10^9$ and positive integers $M$ and $t,$ there is a graph $G$ on $M$ vertices such that 
    \begin{enumerate}[label=\alph*)]
        \item \label{edge-distribution} For all disjoint subsets $S, T \subseteq V(G)$ with $ |S|, |T| \ge \frac{M}{d^{1/3}}$, we have 
        
        \[ \big|e_G(S, T) - \frac{d}{M}|S||T|\big| \le \frac{1}{2} \frac{d}{M}|S||T|, \] 
        and 
        \item \label{max-degree} the maximum degree of $G$ is at most $2d$.
    \end{enumerate}
\end{lemma}
\begin{proof}
    Let $H \sim G(M, d/M),$ that is, $H$ is a random graph on $M$ vertices where every possible edge is present with probability $d/M$ independently. Let $G$ be the graph obtained from $H$ by removing all edges incident to vertices of degree greater than $2d$. Thus, $G$ satisfies \ref{max-degree} deterministically. The expected number of edges removed from $H$ to obtain $G$ is at most
    \begin{align*} \sum_{\{u,v\}\subseteq V(H)} &\Pr[uv \in E(H)] \cdot \Pr[\max\{d_H(u), d_H(v)\} > 2d \, \vert \, uv \in E(H)]\\ 
    &\le \binom{M}{2} \frac{d}{M} \cdot 2 \Pr[\Bin(M-2, d/M) \ge 2d] \le 2Md e^{-d/3} < M,
    \end{align*}
    where we used a standard Chernoff bound (e.g.~Corollary~2.3~in~\cite{janson2011random}) and that $d \ge 10^9.$ By Markov's inequality, with probability at least $3/4,$ we have $e(H) - e(G) \le 4M.$ For fixed disjoint sets of vertices $S$ and $T$ of size at least $\frac{M}{d^{1/3}}$ using the same form of the Chernoff bound, we have
    \begin{align*}  
        \Pr\left[ \big|e_H(S, T) - \frac{d}{M}|S||T|\big| > \frac{1}{4} \frac{d}{M} |S||T| \right] \le 2e^{-\frac{d}{M} |S||T| / 48} \le 2e^{-Md^{1/3} / 48} < 2e^{-20M}.
    \end{align*}
    Taking a union bound over all sets $S, T$ as above, with probability at least  $1 - 2^M \cdot 2^M \cdot 2e^{-20M} > 3/4$, we have 
    \begin{equation} \label{eq:H-edge-distr}
        \big|e_H(S, T) - \frac{d}{M}|S||T|\big| \le \frac{1}{4} \frac{d}{M} |S||T|, \, \forall S, T \subseteq V(H), S \cap T = \emptyset, |S|, |T| \ge \frac{M}{d^{1/3}}.
    \end{equation}
    By a union bound, with probability at least $1/2$ we have $e(H) - e(G) \leq 4M$ and \eqref{eq:H-edge-distr}. Noting that $4M \le \frac{1}{4} \frac{d}{M} \left( \frac{M}{d^{1/3}} \right)^2,$ it follows that $G$ satisfies \ref{edge-distribution} with probability at least $1/2,$ finishing the proof.
\end{proof}

\begin{proof}[Proof of Theorem~\ref{thm:lower-bound}]
 Fix $k \ge 3,$ let $n$ be a large enough integer, set $d = 10^{20k}$ and let $G$ be a  graph on $n_k = dn$  vertices satisfying \ref{edge-distribution}~and~\ref{max-degree} for $d$ whose existence is given by Lemma~\ref{lem:edge-distribution}. Let $H_k = H(G, k).$ We will show that there is no monochromatic copy of $H_k$ in $\phi^{(k)}_n.$ This would prove the theorem since, by construction, $e(H_k) \le n_k \cdot \#\{\text{paths of length } k-2 \text{ in  } G\} \le n_k^2 (2d)^{k-2} = O(n^2)$. We make repeated use of the following lemma.

    \begin{lemma} \label{lem:lower-bound}
        Let $U \subseteq V(G), |U| \ge n_k / (1000^{k-1})$ and let $\ell \ge 3$ be an integer. Denote $H = H(G[U], \ell)$ and suppose there is a monochromatic copy of $H$ in $\phi^{(\ell)}_n.$ Then, there exists a set $U' \subseteq U$ such that $|U'| \ge |U| / 1000$ and there exists a monochromatic copy of the $(\ell-1)$-uniform hypergraph $H' = H(G[U'], \ell-1)$ in $\phi^{(\ell-1)}_n.$
    \end{lemma}
    
    First, let us finish the proof of Theorem~\ref{thm:lower-bound} given Lemma~\ref{lem:lower-bound}. By repeated uses of the lemma, it follows that there are subsets $V(G) = U_k \supseteq U_{k-1} \supseteq \dots \supseteq U_2$ such that there is a monochromatic copy of $H(G[U_\ell], \ell)$ in $\phi^{(\ell)}_n$ for all $2 \le \ell \le k$ and $|U_\ell| \ge |U_{\ell+1}| / 1000$ for all $2 \le \ell \le k-1.$ Hence, we have $|U_2| \ge n_k / 1000^k = 10^{20k} n / 1000^k > n$ and a monochromatic copy of $H(G[U_2], 2)$ in $\phi^{(2)}_n.$ Recall that by definition, $H(G[U_2], 2)$ is a clique on $|U_2| > n$ vertices, hence there is no monochromatic copy of $H(G[U_2], 2)$ in $\phi^{(2)}_n,$ a contradiction.
\end{proof}

\emph{Proof of Lemma~\ref{lem:lower-bound}.}
    Let $s = |U| = |V(H)|,$ let $X = \{x_1, \dots, x_s\} \subseteq \{0, \dots, N_\ell - 1\},$ where $x_1 < \dots < x_s$, form a mononchromatic copy of $H$ and denote by $\Psi \colon V(H) \rightarrow \{ 0, \dots, N_\ell-1 \}$ the given monochromatic embedding.
    
    \begin{claim} \label{claim:Y}
        There is a set $Y = \{y_1, \dots, y_t\} \subseteq X$ of size $t \ge s / 200$ where $y_1 < \dots < y_t$ such that $\delta(\{y_1, \dots, y_t\})$ is a monotone sequence.
    \end{claim}
    
    First we finish the proof of the Lemma given Claim~\ref{claim:Y}. Let $Y \subseteq X$ with $|Y| = t \ge s / 200$ be given by Claim~\ref{claim:Y} and assume that $\delta(Y)$ is increasing, the other case being analogous. Let $L = \{\Psi^{-1}(y_1), \dots, \Psi^{-1}(y_{t/2})\} \subseteq U$ and let $U' \subseteq U$ be the set of all vertices $\Psi^{-1}(y_j)$ with $t/2 < j \le t$ which have at least one neighbour in $L$. We have that $|U'| \ge t / 4,$ as otherwise there is a set of $t/4$ vertices with no edges toward $L,$ contradicting \ref{edge-distribution} since 
    \[ |L| \ge t/4 \ge s / 800 \ge n_k / (1000^k) > n_k / (10^{20k / 3}) = n/d^{1/3}. \]
    
    Let us verify that $\phi_n^{(\ell-1)}$ contains a monochromatic copy of $H' = H(G[U'], \ell-1).$ Let $z_1, z_2, \dots, z_{t'}$ be the elements of $\Psi(U')$ in increasing order and recall that $y_i < z_j$ for all $i \in [t/2], j \in [t'].$ Denote $a_1 = \delta(y_1, z_1)$ and $a_i = \delta(z_{i-1}, z_i)$ for $2 \le i \le t'.$ We will show that the set $\cA = \{a_1, \dots, a_{t'}\}$ forms a monochromatic copy of $H'$ in $\phi^{(\ell-1)}_n$ with the natural correspondence $\Psi' \colon U' \rightarrow \cA$ defined by $\Psi'(\Psi^{-1}(z_i)) = a_i$ for all $i \in [t']$. We do so by showing that, for an edge $e \in E(H'),$ the color $\phi^{(\ell-1)}_n(\Psi'(e))$ is inherited from the color of $\phi_n^{(\ell)}(\Psi(f))$ of some edge $f \in E(H).$
    
    By monotonicity of $\delta(\{y_1, \dots, y_t\}),$ using \ref{prop:maximum}, we have $\delta(z_i, z_j) = a_j$ for any $1 \le i < j \le t'$ and $\delta(y_i, z_j) = a_j$ for any $i \in [t/2]$ and $j \in [t'].$ Now, consider an arbitrary edge $e = \{\Psi^{-1}(z_{j_1}), \dots, \Psi^{-1}(z_{j_{\ell-1}})\} \in E(H'),$ where $j_1 < j_2 < \dots < j_{\ell-1}.$ By construction, some $\ell-2$ of these vertices form a path $P'$ in $G$. By definition of $U',$ any vertex on this path, in particular one of its endpoints, has a neighbour $L$. So, we can attach a vertex $w \in L$ to one of the endpoints of $P'$ to obtain a path on $\ell-1$ vertices in $G$. Hence, $f = e \cup \{w\}$ is a set of $\ell$ vertices, some $\ell-1$ of which form a path in $G$, implying that $f$ is an edge of $H$. Note that $\delta(\Psi(f)) = (a_{j_1}, a_{j_2}, \dots, a_{j_{\ell-1}}),$ which is an increasing sequence. If $\ell \ge 4,$ we have $\phi^{(\ell)}_n(\Psi(f)) = \phi^{(\ell-1)}_n(\delta(\Psi(f))) = \phi^{(\ell-1)}_n(\Psi'(e)).$ In the case $\ell = 3,$ if $\phi^{(2)}_n(\Psi'(e)) = \mathrm{red},$ then $\phi^{(\ell)}_n = C_1$, and if $\phi^{(2)}_n(\Psi'(e)) = \mathrm{blue}$, then $\phi^{(\ell)}_n = C_2$. In either case, it follows that $\cA$ forms a monochromatic copy of $H'$ in $\phi_n^{(\ell-1)},$ as needed. \qed

    \begin{proof}[Proof of Claim~\ref{claim:Y}]
    Consider the following procedure. Start with $Z = X.$ At each step, let $q$ be the largest integer such that not all elements of $Z$ have the same bit at position $q.$ Consider the partition $Z = Z_0 \cupdot Z_1,$ where $Z_p$ denotes the set of elements $z \in Z$ with $\bit(z, q) = p.$ Then, let $Z$ be the larger of $Z_0, Z_1$ and continue the procedure. Eventually we reach a point where $s/4 \le |Z| \le s/2,$ where the lower bound follows since $|Z|$ drops by a factor of at most $2$ in each step. Let $Z^*$ denote the final set $Z$ and let $q^*$ be the last value of $q$ before this point. Then, for all distinct $u, v \in Z^*, w \in X \setminus Z^*,$ we have $\delta(u, v) < q^* \le \delta(u, w).$ Also, note that the elements of $Z^*$ form an interval in the ordered set $X.$ Indeed, at each step all elements in $Z_0$ are smaller than all elements of $Z_1,$ since all elements in $Z$ have the same bit on all positions larger than $q.$ Hence, if $Z$ was an interval in the ordered set $X$ before step $i,$ it is also an interval after step $i$. We shall assume that at least $s/4$ vertices in $X \setminus Z^*$ are smaller than all elements of $Z^*,$ the other case being analogous. Let $W$ denote the set of elements in $X \setminus Z^*$ smaller than every element of $Z^*.$ Now, let $A = \Psi^{-1}(W) \subseteq U$ and let $B$ be the set of all vertices in $\Psi^{-1}(Z^*) \subseteq U$ that have at least one neighbour in $A.$ By \ref{edge-distribution}, it follows that
    \[ |B| \ge |\Psi^{-1}(Z^*)| / 2 \ge s / 8, \]
    as otherwise we obtain a set of $s/8$ vertices with no edge towards $A,$ which is a contradiction since $|A| \ge s/8 \ge n_k / 1000^k > n_k / d^{1/3}.$
    
    Now, we analyse the set $S_1 \coloneqq \Psi(B)$ using a similar procedure as above in steps $i=1, \dots$ At the beginning of step $i,$ we have a set $S_i$ of size at least $2.$ Let $q_i$ be the largest integer such that not all elements of $S_i$ have the same bit at position $q_i.$ Let $S_i = S^0_i \cupdot S^1_i,$ where $S^p_i$ consists of the elements $z \in S_i$ with $\bit(z, q_i) = p.$ Let $p_i \in \{0,1\}$ be such that $|S^{p_i}_i| \ge |S^{1-p_i}_i|.$ Let $S_{i+1} = S^{p_i}_i.$ If $|S_{i+1}| < s / 100,$ stop the process, otherwise continue to step $i+1.$
    
    Assume first that the procedure runs for at least $s / 100$ steps. Then, there is a set $I, |I| \ge s / 200$ and $p \in \{0, 1\}$ such that for all $i \in I,$ we have $p_i = p$ and so $S_{i+1} = S^p_i.$ For each $i \in I,$ let $y_i$ be an arbitrary element in $S^{1-p}_i$ and let $Y = \{y_i, \vert \, i \in I\}.$ Using \ref{prop:maximum}, we have $\delta(y_i, y_{i'}) = q_i$ for any $i, i' \in I, i<i'.$ Moreover, the sequence $(y_i)_{i\in I}$ is increasing if $p=1,$ and decreasing otherwise. Observing that $q_i > q_{i'}$ for $i < i',$ it follows that $Y$ is the desired set.
    
    Therefore, we may assume that the above procedure runs for $h < s / 100$ steps and we will show that this leads to a contradiction. First, we require the following claim.
    \begin{claim} \label{cl:index-i}
        There exists $i \in [h]$ such that $|S^{1-p_i}_i| \ge 2$ and there is a path $P$ of length $\ell-2$ in $G$ with an endpoint $v \in \Psi^{-1}(S^{1-p_i}_i)$ and having its remaining vertices in $\Psi^{-1}(S^{p_i}_i).$
    \end{claim}
        \begin{proof}[Proof of Claim~\ref{cl:index-i}]
         Suppose the claim is not true and let $Q = \bigcup_{i \in [h], |S^{1-p_i}_i| \ge 2} \Psi^{-1}(S^{1-p_i}_i)$ and $T_0 = \Psi^{-1}(S_{h+1}),$ where $S_{h+1}$ is the final set after halting the procedure. Note that $|T_0| \ge s / 200.$ We repeatedly remove from $T_0$ vertices that have fewer than $\ell$ neighbours in $G$ in the current set $T_0.$ Let $T$ denote the final set after these deletions. Then, $|T| \ge s / 400 \ge n_k / (1000)^k,$ as otherwise at the point when we removed half of the vertices, we have two sets of size $q \ge s/400 \ge n / d^{1/3}$ with at most $\ell q$ edges between them, contradicting that $G$ satisfies \ref{edge-distribution}. Using \ref{edge-distribution} again, it follows that there is an edge $vu$ with $v \in Q$ and $u \in T$ since $|Q| \ge |S_1| - |S_{h+1}| - h \ge s/8 - s/100 - s/100 > s/16 >  n / d^{1/3}.$ Since $G[T]$ has minimum degree at least $\ell,$ we can extend this edge to a path of length $\ell-2$ using only vertices in $T.$ Let $i$ be the index such that $v \in \Psi^{-1}(S^{1-p_i}_i).$ Note that $S_j^{p_j} \subseteq S_{j-1}^{p_{j-1}}$ for all $2 \le j \le h$ and $T \subseteq \Psi^{-1}(S_{h+1}) = \Psi^{-1}(S_h^{p_h}).$ It follows that $T \subseteq \Psi^{-1}(S_i^{p_i}),$ so the abovementioned path indeed has all vertices but the first in $\Psi^{-1}(S_i^{p_i}).$ By definition of $Q,$ we also have $|S^{1-p_i}_i| \ge 2,$ as needed.
    \end{proof}
    
    Let $i, v, P$ be given by Claim~\ref{cl:index-i} and let $w$ be an arbitrary vertex in $\Psi^{-1}(S^{1-p_i}_i)$ distinct from $v$. We will show that then $\Psi$ is not a valid embedding, that is, we will find two edges of $H$ whose images get different colors. Let $e = P \cup \{w\} \in E(H).$ We now find another edge $f \in E(H)$ whose image under $\Psi$ gets a different color than $e.$ 
    
    Consider first the case $\ell = 3.$ Then, the path $P$ consists of a single edge $vu$ for some $u \in \Psi^{-1}(S^{p_i}_i).$ Let $u'$ be an arbitrary vertex in $S^{p_i}_i$ distinct form $u,$ which clearly exists since $|S^{p_i}_i| \ge |S^{1-p_i}_i| \ge 2,$ and let $f = \{ v,u,u' \}.$ Note that, by construction, $\delta(u, u'), \delta(v, w) < q_i,$ while $\delta(z, z') = q_i$ for any $(z, z') \in S_i^{p_i} \times S_i^{1-p_i}$. It follows that if $p_i = 1,$ then $\delta(\Psi(e))$ is increasing, while $\delta(\Psi(f))$ is decreasing whereas if $p_i = 0,$ then $\delta(\Psi(e))$ is decreasing, while $\delta(\Psi(f))$ is increasing. In either case, $\Psi(e)$ and $\Psi(f)$ are colored differently by $\phi_n^{(3)}$, as claimed.

\begin{figure}[H]
    \centering
    \begin{subfigure}{\textwidth}
      \centering
      \includegraphics[scale=0.85]{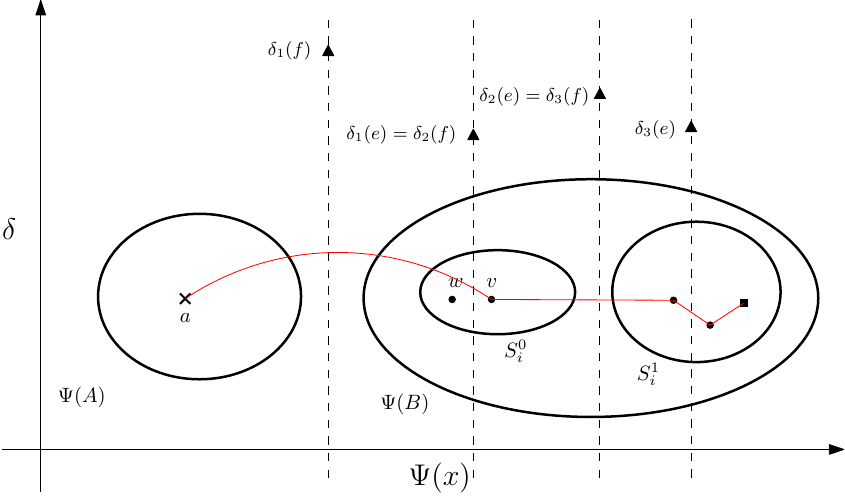}
      \caption{Case $p_i=1$}
      \label{fig:sub1}
    \end{subfigure}%
    \\
    \vspace{0.5cm}
    \begin{subfigure}{\textwidth}
      \centering
      \includegraphics[scale=0.85]{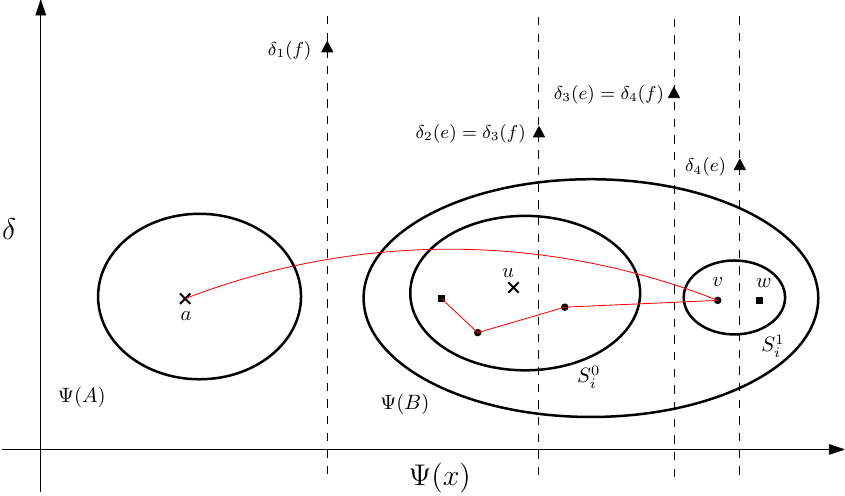}
      \caption{Case $p_i=0$}
      \label{fig:sub2}
    \end{subfigure}
    \caption{The two chosen edges in the case $\ell=5$. Sets $\Psi(A), \Psi(B), S_i^0$ and $S_i^1$ are depicted by ovals. The vertices appearing in both $e$ and $f$ are depicted by points, the vertices in $e \setminus f$ by squares and the vertices in $f \setminus e$ by crosses. The vertices further to the right are mapped by $\Psi$ to larger values. The black triangles correspond to the value of $\delta$ of consecutive vertices, where higher triangles represent larger bit positions. The red lines represent the edges of the corresponding vertices in $G$.}%
    \label{fig:edges}
\end{figure}
    
    Now, consider $\ell \ge 4$ and see Figure~\ref{fig:edges} for an illustration. Recall that $e = P \cup \{w\},$ so $e$ has exactly two vertices in $\Psi^{-1}(S^{1-p_i}_i)$ and the other vertices are in $\Psi^{-1}(S^{p_i}_i).$ Let $\delta = (\delta_1, \dots, \delta_{\ell-1}) = \delta(\Psi(e)).$ Then, $\argmax_{j \in [\ell-1]} \delta_j = 2$ if $p_i = 1$ and $\argmax_{j \in [\ell-1]} \delta_j = k-2$ if $p_i = 0.$ In either case, $\phi_n^{(\ell)}(\Psi(e)) = C_2.$ We will find an edge $f \in E(H)$ whose image receives color $C_1$. Recall that every vertex in $B \supseteq \Psi^{-1}(S^{1-p_i}_i)$ has a neighbour in $A$. Hence, we can extend $P$ by attaching a vertex $a \in A$ to $v$ and then remove its last vertex (which is in $\Psi^{-1}(S_i^{p_i})$) to obtain a path $P'$ of length $\ell-2$ whose first vertex is $a \in A,$ the second vertex is $v \in \Psi^{-1}(S^{1-p_i}_i)$ and the remaining vertices are in $\Psi^{-1}(S^{p_i}_i).$
    
    If $p_i = 1,$ then let $f = V(P') \cup \{w\} \in E(H).$ Consider $\delta = (\delta_1, \dots, \delta_{\ell-1}) = \delta(\Psi(f)).$ Since $f$ has one vertex in $A$ and the rest are in $B,$ it follows that $\argmax_{j \in [\ell-1]} \delta_j = 1.$ Additionally, $\Psi(f)$ has two vertices in $S_i^0$ and the remaining ones are in $S_i^1.$ Hence, $\delta_2 < q_i = \delta_3,$ so $\delta$ is not monotone, implying that $\phi_n^{(\ell)}(\Psi(f)) = C_1.$
    
    If $p_i = 0,$ then let $u$ be an arbitrary vertex in $\Psi^{-1}(S^0_i) \setminus V(P'),$ which exists since $|S^0_i| \ge s / 100 \ge \ell.$ Let $f = V(P') \cup \{ u\} \in E(H)$ and denote $\delta = (\delta_1, \dots, \delta_{\ell-1}) = \delta(\Psi(f)).$ As before, we have $\argmax_{j \in [\ell-1]} \delta_j = 1.$ The largest element of $\Psi(f)$ is $\Psi(v) \in S^1_i,$ while the second and third largest elements are in $S^0_i.$ Hence, $\delta_{\ell-1} > \delta_{\ell-2} < \delta_1,$ which gives $\phi_n^{(\ell)}(\Psi(f)) = C_1.$    
    \end{proof}

\section{Concluding remarks}

There are many remaining interesting problems on Ramsey numbers of hypergraphs. Maybe most notably, while for four or more colors we have lower bound constructions on hypergraph Ramsey numbers for cliques and certain other hypergraphs essentially matching the upper bounds, the bounds are still far apart for two and three colors. It would be interesting to close the gap in these cases. 

Another well-studied question is to bound the $q$-color Ramsey number of bounded degree $k$-uniform hypergraphs on $n$ vertices. It is known that there is a constant $c = c(k, q, \Delta)$ such that $r(H;q) \le c(k, q, \Delta) n$ for any $n$-vertex $k$-uniform hypergraph $H$ with maximum degree at most $\Delta$ and the main question is to understand the value of the factor $c(k, q, \Delta)$ as a function of the maximum degree. In the graph case with two colors, the best lower bound is $c(2, 2, \Delta) = 2^{\Omega(\Delta)}$ due to Graham, R\"{o}dl and Ruci\'{n}ski~\cite{graham2000graphs}, while the best upper bound is $c(2, 2, \Delta) < 2^{O(\Delta \log \Delta)}$ due to Conlon, Fox and Sudakov~\cite{conlon2012twoproblems}. For more than two colors the known upper bound proved in \cite{FS} is much worse and is of the form $c(2, q, \Delta) \le 2^{c_q \Delta^2}$. Turning to hypergraphs, \cite{conlon2009ramsey} showed $c(3, q, \Delta) \le \tw_3(c' \Delta \log \Delta)$ and $c(k, q, \Delta) \le \tw_k(c' \Delta)$ for $k \ge 4,$ where $c'$ is a constant depending on $k$ and $q$. Moreover, they constructed, for any positive integer $\Delta$, a $3$-uniform $n$-vertex hypergraph $H$ with maximum degree $\Delta$ and $r(H; 4) \ge \tw_3(c' \Delta) n$ for some absolute constant $c'.$ The hypergraph we constructed in the proof of Theorem~\ref{thm:lower-bound} has $n$ vertices, maximum degree $\Delta \le c_k n$ and $4$-color Ramsey number at least $\tw_k(c'_k n) n,$ so it can be viewed as a generalization of the aforementioned result to larger uniformities. These results show that in general, it is necessary to have $c(k, 4, \Delta) \ge \tw_k(c'_k \Delta).$ 
However, both our construction and that in \cite{conlon2009ramsey} have $\Delta = \Theta(n)$ and it would be interesting to find a construction which works for any $\Delta$ and sufficiently large $n$ similar to 
the abovementioned lower bound of Graham, R\"{o}dl and Ruci\'{n}ski which works for any $\Delta$ and $n \gg \Delta.$ 

Finally, we think it would be interesting to study the following generalization of Ramsey numbers. For a $k$-uniform hypergraph $H$ and positive integers $N$ and $q$ with $q \leq e(H)$, let $f(N,H,q)$ be the minimum number $r$ such that in every $r$-coloring of the edges of $K_N^{(k)}$ there is a copy of $H$ receiving fewer than $q$ colors. The case $q=2$ is just the inverse (as a function of the number of colors) of the Ramsey number of $H$. The case $H$ is a clique was introduced by Erd\H{o}s and Gyarfas and has been well-studied (see for example \cite{conlon2015recent}).

\end{document}